\documentclass[]{amsart}

\usepackage{epstopdf}

\usepackage[numbers,sort&compress]{natbib}
\bibpunct[, ]{[}{]}{,}{n}{,}{,}
\makeatletter
\def\NAT@def@citea{\def\@citea{\NAT@separator}}
\makeatother

\theoremstyle{plain}
\newtheorem{theorem}{Theorem}[section]

\newtheorem{corollary}[theorem]{Corollary}
\newtheorem{proposition}[theorem]{Proposition}

\theoremstyle{definition}

\theoremstyle{remark}
\newtheorem{remark}{Remark}

    \usepackage[pdftex]{hyperref}
\newcommand{\norm}[1]{\left\Vert#1\right\Vert}
\newcommand{\scal}[1]{\left<#1\right>}

\newcommand{\R}{\mathbb{R}}      

\newcommand{\C}{\mathbb{C}}

\newcommand{\bz}{\overline{z}} \newcommand{\bw}{\overline{w}} \newcommand{\bxi}{\overline{\xi}}

\newcommand{\magn}{\nu}
\newcommand{\gauss}{\mu}
\newcommand{\Lnur}{L^{2,\magn}(\R)}

\newcommand{\Lnuc}{L^{2,\magn}(\C)}

\title[Non-trivial 1d and 2d integral transforms of Segal-Bargmann type]
 {Non-trivial 1d and 2d integral transforms of Segal-Bargmann type } 

\author{Abdelhadi Benahmadi} \email{abdelhadi.benahmadi@gmail.com / ag@fsr.ac.ma}
\author{Allal Ghanmi}     
 \address{CeReMAR, A.G.S., L.A.M.A., Department of Mathematics, P.O. Box 1014,
             Faculty of Sciences, Mohammed V University in Rabat, Morocco}
 \thanks{Dedicated to the Memory of Kettani Ghanmi and Mohammed Hamdi}
           
\begin{document}
\maketitle


\begin{abstract}
Generating functions for the univariate complex Hermite polynomials (UCHP) are employed to introduce some non-trivial one and two-dimensional integral transforms of Segal-Bargmann type in the framework of specific functional Hilbert spaces. The approach used is issued from the coherent states framework.
Basic properties of these transforms are studied. Connection to some special known transforms like Fourier and Wigner transforms is established.
The first transform is a two-dimensional Segal-Bargmann transform whose kernel is related the exponential generating function of the UCHP. The second one connects the so-called generalized Bargmann-Fock spaces (or also true-poly-Fock spaces in the terminology of Vasilevski) that are realized as the $L^2$-eigenspaces of a specific magnetic Schr\"odinger operator.
Its kernel function is related to a Mehler formula of the UCHP.
%
%
 35A22; 44A15;  33C45; 42C05; 42A38; 32A10
\end{abstract}

\section{Notations and motivations}
We denote by $L^{2,\magn}(X)$ the Hilbert space of all square integrable functions on $X=\R,\R^2\simeq\C,\C^2$ with respect to the Gaussian measure $d\lambda_\magn(s) := e^{-\magn |s|^2}d\lambda(s)$, where $d\lambda$ is the Lebesgue measure on $X$ with $d\lambda(s)=dx,dx_zdy_z,dx_zdy_zdx_wdy_w$ for $s=x \in \R$, $s=z=x_z+iy_z\in \C\simeq \R^2$ and $s=(z,w) \in \C^2$, respectively. We denote by $\mathcal{F}^{2,\magn}(X)$ the Bargmann-Fock space constituted of all holomorphic functions on $X$, when $X=\C$ or $X=\C^2$, belonging to  $L^{2,\magn}(X)$,
$$\mathcal{F}^{2,\magn}(X)= \mathcal{H}ol(X) \cap L^{2,\magn}(X).$$

Part of the motivation of the present work comes from the recent investigations in the theory of the UCHP
which have played a crucial rule in obtaining our transforms. Such polynomials have been employed in many branches of Mathematics, including Markov process, the nonlinear analysis of travelling wave tube amplifiers, signal processing, the singular values of the Cauchy transform,
coherent states theory, combinatorics and the distribution of zeros of the automorphic reproducing kernel function.
 For basic properties on these polynomials we refer to \cite{Gh08JMAA,ABEG2015,IsmailZhang2017,Gh2017Mehler,ElfardiGIS2018}.

In fact, we employ generating functions for the UCHP in order to obtain non-trivial one and two-dimensional integral transforms of Segal-Bargmann type in the framework of specific functional Hilbert spaces and study their basic properties, including the connection to some special known transforms like Fourier and Wigner transforms.

\section{Introduction} \label{s1}

The one-dimensional Segal-Bargmann transform \cite{Bargmann1961,Segal62} establishes an isometric isomorphism from the Hilbert space $\Lnur$ onto the Bargmann-Fock space $\mathcal{F}^{2,\magn}(\C)$. 
 It is specified with the formula (with a slightly different convention from the classical one)
  \begin{eqnarray}\label{IntTransf20}
\mathcal{B}^{1,\magn}(\phi)(z) :=
\left(\frac{\magn}{\pi}\right)^{3/4} \int_{\R} \rho_0^\magn \left(x  -  \frac{z}{\sqrt{2}}\right) \phi(x)  dx,
 \end{eqnarray}
 where $\rho_0^\magn(\xi)=e^{-\magn \xi^2}$ is the analytic continuation to $\C$ of the standard Gaussian density on $\R$.
 Such transform has found many applications in quantum optics, in signal processing and in harmonic analysis on phase space. A nice overview of its properties and applications can be found in \cite{Folland1989,Neretin1972}.
 Many generalizations have been considered in the literature including the Hall's transforms for compact Lie groups \cite{Hall1994,GrossMalliavin1996,MouraoNQ2017} as well as the so-called generalized Segal-Bargmann transform of level $n$ (see \cite{Mouayn2011} where there $\magn=1$),
 \begin{eqnarray}\label{IntTransf2}
\mathcal{B}^{1,\magn}_n(\phi)(z) := \left(\frac{\magn}{\pi}\right)^{3/4}  \left(\frac{1}{2^n\magn^n n! }\right)^{1/2}
\int_{\R} \rho_0^\magn \left(x  -  \frac{z}{\sqrt{2}}\right) H_n^{\magn} \left( \frac{z + \bz }{\sqrt{2}}  -  x   \right) \phi(x)  dx ,
 \end{eqnarray}
where $H_n^\magn(x) $; $\magn >0$, is the $n^{\mbox{th}}$ rescaled real Hermite polynomial defined by
\begin{eqnarray} \label{wrHn}
H_n^\magn(x) := (-1)^n e^{{\magn}x^2} \frac{d^n}{dx^n}\left(e^{-\magn x^2}\right),
\end{eqnarray}
 so that one recovers the standard Segal-Bargmann transform in \eqref{IntTransf20} when taking $n=0$.
Up to a multiplicative constant, the transform $\mathcal{B}^{1,\magn}_n$ coincides with the isometric operator considered by Vasilevski and linking the space of square integrable functions on the real line with the so-called true-poly-Fock spaces \cite[Theorem 2.5]{Vasilevski2000} (see also \cite{Abreu2010,AbreuGossonMouayn2015}).
 The considered true-poly-Fock spaces can in fact be realized as the $L^2$-eigenspaces
 \begin{eqnarray}\label{GBFn}
\mathcal{F}^{2,\magn}_n(\C)= \{ f\in \Lnuc ; \quad \Delta_\magn f = n f  \}
\end{eqnarray}
of the magnetic Schr\"odinger operator
 $$
\Delta_\magn := - \dfrac{\partial^2}{\partial z\partial\bz  } + \magn \bz  \dfrac{\partial}{\partial \bz }.
$$
The kernel function in \eqref{IntTransf2} is in fact the exponential generating function involving the product of the real $H_m^\magn(x) $ in \eqref{wrHn} and the univariate complex Hermite polynomials (UCHP)
 \begin{eqnarray}\label{gchpmu}
         H_{m,n}^\magn  (z,\bz )=(-1)^{m+n}e^{\magn  z\bz }\dfrac{\partial ^{m+n}}{\partial \bz^{m} \partial z^{n}} \left(e^{-\magn  z \bz }\right).
 \end{eqnarray}
We have to reconsider this in Section 2 and provide a direct and simpler proof of a general form of such exponential generating function (Theorem \eqref{thm:bilgenfct}).
The incorporation of the parameter $\magn $ is fairly interesting for its physical meaning. In fact, it can be interpreted as the magnitude of a constant magnetic field applied perpendicularly to the Euclidean plane.

 The standard Hilbert space $\Lnur$ on the real line can also be connected to the two-dimensional Bargmann-Fock space $\mathcal{F}^{2,\magn}(\C^2)$
 by considering the special one-to-one transform
 $$(\mathcal{G}^\nu\varphi)(z,w) = \left(\frac{\nu}{\pi}\right)^{\frac{1}{2}} \mathcal{C}_\psi (\mathcal{B}^{1,\magn}\varphi)(z,w)$$
  obtained as the composition operator $\mathcal{C}_\psi f = f\circ \psi$ of the one-dimensional Segal-Bargmann transform $\mathcal{B}^{1,\magn}$ with the specific symbol $ \psi(z,w) = \frac{z+iw}{\sqrt{2}}.$
  Its image is characterized as
\begin{equation}\label{Image}
\mathcal{G}^\nu (\Lnur) = \left\{f\in \mathcal{F}^{2,\nu}(\C^2); \, \left(\frac{\partial}{\partial z} + i \frac{\partial}{\partial w}\right) f =0 \right\}=: E^{2,\magn}_{+}(\C^2),
\end{equation}
which can also be identified to the space of slice (left) regular functions on the quaternions leaving invariant the slice $\C_i\simeq \C$ (see \cite{DG2018b} for details). Moreover, this transform can be realized as $\mathcal{G}^\nu = \mathcal{B}^{2,\magn}\circ \mathcal{B}^{1,\magn}$, where $\mathcal{B}^{1,\magn}$ is as in \eqref{IntTransf20} and $\mathcal{B}^{2,\magn}$ is the two-dimensional Segal-Bargmann transform given by
\begin{eqnarray} \label{2dSBT}
\mathcal{B}^{2,\magn} \psi (z,w)= \left(\frac{\magn}{\pi}\right)^{\frac{3}{2}}
  \int_{\mathbb{R}^2} \rho_0^\magn \left(x  -  \frac{z}{\sqrt{2}}\right)\rho_0^\magn \left(y  -  \frac{w}{\sqrt{2}}\right)
   \psi(x,y)dxdy,
   \end{eqnarray}
and makes the quantum mechanical configuration space $\Lnuc$ unitarily isomorphic to the phase space $\mathcal{F}^{2,\magn}(\C^2)$. Its kernel function is the product of two copies of the kernel function of the  one-dimensional Segal-Bargmann transform $\mathcal{B}^{1,\magn} $ and therefore to the generating function of the product $H_m^\magn(x)H_n^\magn(y)$.

\section{Statement of main results} \label{s2}

In the present paper, we introduce two integral transforms in the framework of specific functional Hilbert spaces by means of the generating functions of the UCHP in \eqref{gchpmu}. We study their basic properties and identify their images.
 The first one is a non-trivial two-dimensional integral transform of Segal-Bargmann type. It is defined by
   \begin{equation}\label{IntTransf1}
\mathcal{T}^\magn(\psi)(z,w)
:= \left(\frac{\magn}{\pi}\right)^{3/2} \int_{\C}  e^{-\magn(z-\xi)(w-\overline{\xi})}  \psi(\xi) d\lambda(\xi).
  \end{equation}

\begin{theorem}\label{thm:IntTransf1}
The integral operator $\mathcal{T}^\magn$ in \eqref{IntTransf1} defines an isometric isomorphism from $\Lnuc$ onto the two-dimensional Bargmann-Fock space $\mathcal{F}^{2,\magn}(\C^2)$.
Its inverse is given by
 \begin{eqnarray}\label{InverseT}
\left( \mathcal{T}^\magn \right)^{-1}(\varphi )(\xi) = \left(\frac{\magn}{\pi}\right)^{3/2}
 \int_{\C^2}  e^{-\magn(|z|^2+|w|^2)+\magn(\bz\xi+ \bw\bxi-\bz\bw)}  \varphi(z,w) d\lambda(z,w).
  \end{eqnarray}
Moreover, we have
 \begin{eqnarray}\label{ActionIntTransf1}
 \int_{\C}  e^{-\magn(z-\xi)(w-\overline{\xi})}  H_{m,n}^{\magn }(\xi,\overline{\xi}) d\lambda(\xi)
= \left(\frac{\pi}{\magn}\right) \magn ^{m+n} z^m w^n.
 \end{eqnarray}
\end{theorem}

As immediate consequence of \eqref{ActionIntTransf1}, we get
\begin{corollary}
The identity
\begin{eqnarray}\label{ident}
\int_{\C}  e^{-\magn|\xi|^2}  H_{m,n}^{\magn }(\xi,\overline{\xi}) d\lambda(\xi) =0
 \end{eqnarray}
 holds for every integers $m,n$ such that $mn\geq 1$.
\end{corollary}
 We also recover the specific integral representation of the kernel function $e^{\magn zw }$ (the reproducing property),
\begin{equation}\label{IntegFKernel}
    \int_{\C}  e^{-\magn|\xi|^2 + \magn(  z\overline{\xi} + w \xi )}  d\lambda(\xi) = \left(\frac{\pi}{\magn}\right) e^{\magn zw } .
\end{equation}
Taking into account the orthogonal Hilbertian decomposition $ \Lnuc = \bigoplus_{n=0}^{\infty} \mathcal{F}^{2,\magn}_n(\C)$,
a quite natural question arises of whether the image of $\mathcal{F}^{2,\magn}_{n}(\C)$ by the transform $\mathcal{T}^\magn$ can be
characterized. In fact, the following holds making use of \eqref{ActionIntTransf1}.

\begin{proposition}\label{propimage}
We have
\begin{equation}\label{ImageFnT}
\mathcal{T}^\magn(\mathcal{F}^{2,\magn}_{n}(\C)) = \{ f(z,w) = w^n h(z); \, h\in \mathcal{F}^{2,\magn}(\C) \}.
\end{equation}
\end{proposition}

Interesting integral transforms can be derived as special cases of $\mathcal{T}^\magn$, including those obtained by restriction.
For example, one can consider the restriction operators to the diagonal and the "anti-diagonal" of $\C^2$ or also to $\C\times \{0\}$. They are defined respectively by $\mathcal{R}^\magn_+ \psi  (z):=\mathcal{T}^\magn \psi (z,z)$, $\mathcal{R}^\magn_- \psi  (z):=\mathcal{T}^\magn \psi (z,\bz)$ and $\mathcal{T}^\magn \psi (z,0)$.
The following result shows in particular that the operators $\mathcal{R}^\magn_+ $ and $\left( \mathcal{T}^\magn \right)^{-1} \Gamma_{-i} \mathcal{T}^\magn$, with $\Gamma_{-i}\psi (z,w) := \psi (-i z,-i w)$, reduce further to the shifted Fourier transform $\widetilde{\mathcal{F}}^\magn$ defined on $\Lnuc$ by
   \begin{eqnarray}\label{Fourier}
   \widetilde{\mathcal{F}}^\magn (\varphi) (\xi)
   := \left(\frac{\magn }{2\pi}\right)  \int_{\C}  e^{\frac{\nu}{2}(\xi-iu)(\overline{\xi}-i\overline{u})} \varphi (u)  d\lambda(u).
  \end{eqnarray}
which is the $\Lnuc$ version of the standard Fourier transform $\mathcal{F}^\magn$ on $L^{2,0}(\C)$ with
$\widetilde{\mathcal{F}}^\magn  =  \mathcal{M}_{\nu/2}  \mathcal{F}^\magn \mathcal{M}_{-\nu/2} $. Here $ \mathcal{M}_\alpha$ denotes the multiplication operator (ground state transform) $\mathcal{M}_\alpha f:=e^{-\alpha\vert{z}\vert^2}f$.

 \begin{theorem}\label{thm:IntTransf1b}
We have the formulas $\Gamma_{-i} \mathcal{R}^\magn_+  = \widetilde{\mathcal{F}}^{\nu/2}$ and
   \begin{eqnarray}\label{ConnFourier}
   \left( \mathcal{T}^\magn \right)^{-1} \Gamma_{-i} \mathcal{T}^\magn = \widetilde{\mathcal{F}}^\magn .
    \end{eqnarray}
\end{theorem}

The identity \eqref{ConnFourier} shows that the shifted Fourier transform $\widetilde{\mathcal{F}}^\magn$ in \eqref{Fourier} is imbedded in the reducible representation of the unitary group $U(1):=\{\theta \in \C; \, |\theta |=1 \}$ on the two-dimensional Bargmann-Fock space $\mathcal{F}^{2,\magn}(\C^2)$ defined by $\Gamma_\theta \psi (z,w) := \psi (\theta z,\theta w)$. Moreover, from \eqref{ActionIntTransf1} and \eqref{ConnFourier}, we can prove the following

  \begin{proposition}\label{propFourier}
The UCHP form an orthogonal eigenfunction basis of $\widetilde{\mathcal{F}}^\magn$ with
 \begin{eqnarray}\label{FourierHermite}
 \widetilde{\mathcal{F}}^\magn (H_{m,n}^{\magn }) = i^{m+n} H_{m,n}^{\magn }.
  \end{eqnarray}
\end{proposition}

Notice also that the integral operator $\mathcal{T}^\magn $ is closely related to the two-dimensional Segal-Bargmann transform in \eqref{2dSBT} as well as to the Wigner transform
\begin{equation}\label{Wignerdf} 
 \mathcal{W}^{\magn}(f)(x, y) = \left( \frac{1}{2\pi}\right)^{1/2} \int_{\R} e^{-i\nu xt} f\left(y+\frac{t}{2},y-\frac{t}{2}\right) dt ; \quad f \in L^{2,0}(\R^2).
 \end{equation}
The transform $ \mathcal{W}^{\magn}$ is connected with the phase space formulation of quantum mechanics and Weyl correspondence \cite{Pool1966,Folland1989,Thangavelu1998}.
We consider the standard action of the group of $2 \times 2$ matrices $ M_{2}(\C) $ defined by
$$ g\cdot (z,w) = (az+bw, cz+dw)    ; \, g = \left(\begin{array}{cc}  a & b \\  c & d \\ \end{array} \right) \in M_{2}(\C) $$
on $\C^2$ that we extend to functions on $\C^2$ by considering $\Gamma_{g} f (z,w) : = f (g\cdot (z,w))$,
and notice that
\begin{equation}\label{gi}
g_i := \left( \begin{array}{cc} 1 & i \\ 1 & -i  \end{array} \right)  \in \sqrt{2} U(2),
\end{equation}
 where $U(2)$ is the subgroup of $2 \times 2$ unitary matrices.
The formula \eqref{ConnIntTransf1Wigner} below is in some how the analogue of
$\mathcal{B}^{2} \mathcal{W} =\Gamma_g \mathcal{B}^{2} $ proved by Shun-Long Luo in \cite[Proposition 2]{Shun-LongLuo1998}.

\begin{theorem}\label{ThmBargmannWigner} We have the identities $ \mathcal{B}^{2,\magn}  = \Gamma_{g_i}\mathcal{T}^\magn$
  and
  \begin{eqnarray}\label{ConnIntTransf1Wigner}
    \mathcal{T}^\magn \mathcal{W}^\nu
= \left( \frac{1}{2\magn}\right)^{1/2}  e^{-\frac{\magn}{4}(z+w)^2} \Gamma_{-ig_{i}}\mathcal{T}^{\frac{\magn}2} .
\end{eqnarray}
  \end{theorem}

The second integral transform we deal with is new and connects the generalized Bargmann-Fock spaces $\mathcal{F}^{2,\magn}_n(\C)$ and $\mathcal{F}^{2,\magn}_{n'}(\C)$. It is defined by
 \begin{eqnarray}\label{IntTransf3}
\mathcal{T}^\magn_{n,n'} (\psi)(z) := \left(\frac{(-1)^{n'} \magn}{\pi \sqrt{ n! n'! \magn^{n+n'} }}\right)
\int_{\C} e^{-\magn|\xi|^2 +\magn \bxi z }  H_{n,n'}^\magn(\xi-z,\bxi-\bz)   \psi(\xi)  d\lambda(\xi).
  \end{eqnarray}

 \begin{theorem}\label{thm:IntTransf3}
The integral transform $\mathcal{T}^\magn_{n,n'} $ is a unitary operator from $\mathcal{F}^{2,\magn}_n(\C)$ onto $\mathcal{F}^{2,\magn}_{n'}(\C)$ and its inverse is given by $\left(\mathcal{T}^\magn_{n,n'}\right)^{-1} = \mathcal{T}^\magn_{n',n} $.
  Moreover, we have the following integral reproducing property for the UCHP
    \begin{eqnarray}\label{ActionIntTransf3}
\mathcal{T}^\magn_{n,n'} (H_{m,n}^{\magn } )(z) = \left(\frac{n! \magn^{n}}{n'! \magn^{n'} } \right)^{1/2} H_{m,n'}^{\magn }(z,\bz).
  \end{eqnarray}
\end{theorem}

In particular, the integral operator $\mathcal{T}^\magn_{0,n}$ maps isometrically the standard Bargmann-Fock space $\mathcal{F}^{2,\magn}(\C)$ onto the Bargmann-Fock space $\mathcal{F}^{2,\magn}_{n}(\C)$. Its inverse is given through $\mathcal{T}^\magn_{n,0}$,
 \begin{eqnarray}\label{IntTransf3n0}
\mathcal{T}^\magn_{n,0} (\psi)(z) := \left(\frac{ \magn }{\pi}\right) \left(\frac{ \magn^{n}}{n!}\right)^{1/2}
\int_{\C} e^{-\magn|\xi|^2 +\magn \bxi z }  (\xi-z)^n   \psi(\xi)  d\lambda(\xi).
  \end{eqnarray}
Accordingly, the UCHP $H_{m,n}^{\magn }$, for varying $m$ and fixed $n$, is a common basis of $L^2$-eigenfunctions of the magnetic Schr\"odinger operator $\Delta_\magn$ and the integral operator
\begin{eqnarray}\label{IntTransf3nn}
\mathcal{T}^\magn_{n} \psi(z) := \mathcal{T}^\magn_{n,n} \psi(z) = \int_{\C} e^{-\magn |\xi|^2 + \magn\bxi z }  H_{n,n}^{\magn }(\xi-z,\bxi-\bz)   H_{m,n}^{\magn }(\xi,\bxi)  d\lambda(\xi)
  \end{eqnarray}
thanks to \eqref{ActionIntTransf3} with $n=n'$. The proof of Theorem \ref{thm:IntTransf3} lies essentially on the following result giving the closed expression of the generating function involving $t^n H_{m,n}^{\magn }(z,\bz ) H_{n,m'}^{\magn }(w,\bw ).$

\begin{theorem}\label{thm:bilgenfct1} For every $t$ in the unit circle, $|t|=1$, and $z,w\in \C$, we have the Mehler formula
\begin{eqnarray}\label{genfct1hh}
\sum\limits_{n=0}^{+\infty} \frac{ t^n }{n! \magn ^n  }  H_{m,n}^{\magn }(z,\bz ) H_{n,m'}^{\magn }(w,\bw )
 &= (-t)^{m'}   H_{m,m'}^{\magn}( z -tw, \bz - \overline{t}\bw) e^{\magn  t w\bz } .
\end{eqnarray}
\end{theorem}

\section{Abstract formalism for the coherent states transforms and application} \label{s3}
The proofs of Theorems \ref{thm:IntTransf1} and \ref{thm:IntTransf3} are subject to a general principle issued from the framework of coherent states transform. A brief review of this principle presented in the sequel is taken from \cite{Hall1994,Gazeau2009}.
Let $(\mathcal{H}_X;\omega_X)$ be an infinite dimensional complex functional Hilbert space on $X$ with an orthogonal basis $\left\{
e_{n}\right\}_{n}$ with respect to the inner scaler product
$$ \scal{\phi,\psi}_{\mathcal{H}_X} := \int_X \phi (x) \overline{\psi(x)} \omega_X(x)dx $$
for given weight measure $\omega_X$. In a similar way we consider $(\mathcal{H}_Y;\omega_Y)$ with an orthogonal basis $\left\{
f_{n}\right\}_{n}$ and assume that $\mathcal{H}_Y$ is in addition a reproducing kernel Hilbert space with reproducing kernel $K(y,y')$.
Associated to the data $(\mathcal{H}_X;\omega_X;\left\{e_{n}\right\}_{n})$ and $(\mathcal{H}_Y;\omega_Y;\left\{f_{n}\right\}_{n})$, we perform the following kernel function $T: X\times Y \longrightarrow \C$ defined by
$$ T(x,y) := \sum_{n=0}^{\infty} \frac{\overline{e_{n}(x)} f_{n}(y) }{\norm{e_{n}}_{\mathcal{H}_X}\norm{f_{n}}_{\mathcal{H}_Y}} .$$
 It is straightforward to check that $\scal{T(\cdot,y),T(\cdot,y')}_{\mathcal{H}_X}$ reduces further to $K(y,y')$, the reproducing kernel function of $\mathcal{H}_Y$.
Moreover, the map $y \in Y \longmapsto T(\cdot,y) \in \mathcal{H}_X$ defines a quantization of $Y$ into $\mathcal{H}_X$.
Thus, we can consider the integral transform
$$ \mathcal{T}(\phi)(y) := \int_X T(x,y) \phi(x) \omega_X(x) dx$$
for every $\phi \in \mathcal{H}_X$.  This transform maps $\mathcal{H}_X$ onto $\mathcal{H}_Y$ and satisfies
$$ \mathcal{T}\left( \frac{e_{k}}{\norm{e_{k}}_{\mathcal{H}_X}}\right) = \frac{f_{k}}{\norm{f_{k}}_{\mathcal{H}_Y}} $$
and subsequently
$\mathcal{T}(\phi) = \sum\limits_{n=0}^{\infty} \beta_n   f_{n} \in \mathcal{H}_Y$ for every $\phi = \sum\limits_{n=0}^{\infty} \alpha_n e_n \in \mathcal{H}_X$,
where $\beta_n := \alpha_n {\norm{e_{n}}_{\mathcal{H}_X}}/{\norm{f_{n}}_{\mathcal{H}_Y}} $.
Moreover, it is readily easy to see that
\begin{align*} \norm{ \phi }_{\mathcal{H}_X}^2
=  \sum\limits_{n=0}^{\infty} |\alpha_n|^2 \norm{e_{n}}_{\mathcal{H}_X}^2
=  \sum\limits_{n=0}^{\infty} |\beta_n|^2  \norm{f_{n}}_{\mathcal{H}_Y}^2
= \norm{  \mathcal{T}(\phi)}_{\mathcal{H}_Y}^2 .
\end{align*}
Thereby, $\mathcal{T}$ defines a isometric linear transform from $\mathcal{H}_X$ onto $\mathcal{H}_Y$ and
 the function $x \longmapsto \scal{ \phi , T(x,\cdot)  }_{\mathcal{H}_Y} $ belongs to $\mathcal{H}_X$ for every $\phi \in \mathcal{H}_Y$, since
\begin{align*}
\scal{ \phi , T(x,\cdot)  }_{\mathcal{H}_Y}
&=   \scal{\phi , \sum_{n=0}^{\infty} \frac{\overline{e_{n}(x)} f_{n}}{\norm{e_{n}}_{\mathcal{H}_X}\norm{f_{n}}_{\mathcal{H}_Y}} }_{\mathcal{H}_Y}
=   \sum_{n=0}^{\infty} \frac{ \scal{\phi ,  f_{n} }_{\mathcal{H}_Y} }{\norm{e_{n}}_{\mathcal{H}_X}\norm{f_{n}}_{\mathcal{H}_Y}} e_{n}(x).
  \end{align*}
In addition, we have the following integral representation 
\begin{align*}
\phi(y) &= \int_X T(x,y) \scal{ \phi , T(x,\cdot)  }_{\mathcal{H}_Y}   \omega_X(x)dx
= \mathcal{T} \left( \scal{ \phi , T(x,\cdot)  }_{\mathcal{H}_Y} \right)(y)   \label{resId}
\end{align*}
for every $\phi \in \mathcal{H}_Y$.

The above formalism is then applied in \cite{Mouayn2011} to recover the result in \cite[Theorem 2.5]{Vasilevski2000} that we can reworded as follows
\begin{theorem}\label{thm:IntTransf2}
The integral operator in \eqref{IntTransf2} defines an isometric isomorphism from $\Lnur$ onto the generalized Bargmann-Fock space $\mathcal{F}^{2,\magn}_n(\C)$ defined by \eqref{GBFn}. Moreover, we have
   \begin{eqnarray}\label{ActionIntTransf2}
 \mathcal{B}^{1,\magn}_n (H_{m}^\magn )(z) = \left(\frac{\magn }{\pi}\right)^{1/4}  \left(\dfrac{ 2^m }{ n! \magn^{n}  }\right)^{1/2} H_{m,n}^\magn (z,\bz).
  \end{eqnarray}
\end{theorem}

The proof of Theorem \ref{thm:IntTransf2} lies on a particular case of following result.
\begin{theorem} \label{thm:bilgenfct}
We have the generating function
\begin{eqnarray} \label{thm:bilgenfct3HnHmn}
 \sum_{m=0}^{+\infty}\dfrac{\xi^m H_{m}^\mu (x) \overline{H_{m,n}^{\magn }(z; \bz)} }{ m!  \magn^m }
 =    e^{-\mu ( \xi^2\bz^2 - 2 x \xi \bz ) }
 H_n^{\mu\xi^2} \left( \bz + \frac{\magn  }{2\mu \xi^2 } z - \frac{x}{\xi}   \right).
\end{eqnarray}
\end{theorem}
Here $H_n^\magn(x)$ (resp. $H_{m,n}^\magn(z,\bz)$) are the real (resp. univariate complex) Hermite polynomials in \eqref{wrHn} (resp. \eqref{gchpmu}).
Such polynomials are the rescaled version of the real (resp. complex) Hermite polynomials $H_{n}$ (resp. $H_{m,n}$) corresponding to $\magn =1$in the sense that
 \begin{eqnarray}\label{gchpReduction3}
 H_{m,n}^{\magn}\left(z, \bz \right)= \magn^{\frac{m+n}2} H_{m,n}(\sqrt{\magn }z;\sqrt{\magn }\bz)
 \quad \mbox{and} \quad  \magn^m H_m( \magn x) = H_m^{\magn^2}(x)).
\end{eqnarray}

The generating function \eqref{thm:bilgenfct3HnHmn} is a general form of the one obtained in \cite[Proposition 4.2]{Mouayn2011}.
The proof we present below is different, direct and simpler than one provided in \cite{Mouayn2011}.

 \begin{proof}
Making use of $H_{m,n}(z,\bz)=e^{-\Delta_{\C}} (z^{m}\bz^{n})$ as well as the well-known generating function for the real Hermite polynomials (\cite[p.187]{Rainville71}),
$$\sum_{n=0}^{+\infty}\dfrac{ \xi^{n} H_{n}^\mu (x) }{ n! } = e^{-\mu \xi^2 + 2\mu x \xi},$$
we obtain
\begin{align*}
\sum_{n=0}^{+\infty}\dfrac{t^n H_{n}^\mu (x) H_{m,n}(z,\bz)}{ n!\magn^n }
& = e^{-\Delta_{\C}} \left( z^{m} \sum_{n=0}^{+\infty}\dfrac{ (t\bz/\magn)^{n} H_{n}^\mu (x) }{ n! } \right)
\\&= e^{\mu x^2} e^{-\Delta_{\C}} \left( z^{m} e^{-\mu[\frac{t\bz}{\magn }- x  ]^2 }  \right).
\end{align*}
Now, by utilizing the fact
$$ \dfrac{\partial^j}{\partial \bz^j} \left( e^{- (a\overline{z} - b)^2} \right) =  (-1)^j a^j e^{- (a\overline{z} - b)^2}  H_j( a\overline{z} - b) $$
as well as the well-known identity 
$$ \sum_{j=0}^{m} \binom{m}{j}   H_{j}(x) (2\xi)^{m-j} = H_{m}( x + \xi ),$$
 we can rewrite the above sum as
\begin{align*}
\sum_{n=0}^{+\infty}\dfrac{t^n H_{n}^\mu (x) H_{m,n}(z,\bz)}{ n!\magn^n }
&= e^{\mu x^2}  \sum_{j=0}^{m}\binom{m}{j} z^{m-j} (-1)^j \dfrac{\partial^j}{\partial \bz^j} \left( e^{-\mu[\frac{t\bz}{\magn }- x  ]^2 } \right)
\\ &= \left(\frac{\sqrt{\mu}t}{\magn }\right)^m   e^{-\mu\left[\frac{t^2\bz^2}{\magn ^2}- 2 x \frac{t\bz}{\magn } \right] }
\sum_{j=0}^{m}\binom{m}{j} \left(\frac{\magn }{\sqrt{\mu}t} z \right)^{m-j}  H_j \left( \frac{\sqrt{\mu}t}{\magn } \bz - \sqrt{\mu} x  \right)
\\ &= \left(\frac{\sqrt{\mu}t}{\magn }\right)^m   e^{-\mu\left[\frac{t^2\bz^2}{\magn ^2}- 2 x \frac{t\bz}{\magn } \right] }
 H_m \left( \frac{\sqrt{\mu}t}{\magn } \bz + \frac{\magn }{2\sqrt{\mu}t} z - \sqrt{\mu} x   \right).
\end{align*}
Finally, the desired result follows thanks to \eqref{gchpReduction3}.
\end{proof}

\begin{proof}[of Theorem \ref{thm:IntTransf2}]
The kernel function associated to the Hilbert space $\Lnur$ and the generalized Bargmann-Fock space $\mathcal{F}^{2,\magn}_n(\C)$ is given by
\begin{align*}
T_n^\magn( x; z)
 &: =  \left(\frac{\magn }{\pi}\right)^{3/4}     \left( \frac{1}{n! \magn^{n}} \right)^{1/2}
\sum_{m=0}^{+\infty}\dfrac{H_{m}^\magn (x) H_{m,n}^\magn (z,\bz) }{ \sqrt{2^m} \magn ^m m! }.
\\&: =  \left(\frac{\magn }{\pi}\right)^{3/4}     \left( \frac{1}{n! \magn^{n}} \right)^{1/2}
\sum_{m=0}^{+\infty}\dfrac{H_{m}^\magn (x)\overline{H_{m,n}^\magn (\bz,z)} }{ \sqrt{2^m} \magn ^m m! }.\nonumber
\end{align*}
Indeed, the univariate real Hermite polynomials $ H_m^\magn(x)$ defined by \eqref{wrHn} form an orthogonal basis of $\Lnur$ with norm given explicitly by
$$ \norm{H_m^\magn }_{\Lnur}^2= \left(\frac{\pi}{\magn }\right)^{1/2}  2^m\magn ^m m!.$$
While the UCHP $H_{m,n}^{\magn }(z;\bz)$, for fixed $n$ and varying $m$, is an orthogonal basis of the generalized Bargmann-Fock space of level $n$, $\mathcal{F}^{2,\magn}_n(\C)$. The square norms of $H_{m,n}^{\magn }(z;\bz)$ is given by
   \begin{eqnarray}\label{SqnormHmn}
    \norm{H_{m,n}^\magn}^2_{L^{2,\magn}(\C)}= \left(\frac{\pi}{\magn}\right) m! n! \magn^{m+n}.
   \end{eqnarray}
By means of the generating function \eqref{thm:bilgenfct3HnHmn} we get
\begin{eqnarray}
T_n^\magn( x; z)
  =  \left(\frac{\magn }{\pi}\right)^{3/4} \left(\frac{1}{2^n \magn^n n! }\right)^{1/2}
  e^{-\frac{\magn }{2} z^2 +  \sqrt{2}\magn  x  z } H_n^{\magn} \left( \frac{ z + \bz}{\sqrt{2}}  - x   \right).
  \label{kernelHnHmn}
\end{eqnarray}
The proof follows by the coherent states formalism presented above.
\end{proof}

\section{Proof of main theorems}

\begin{proof}[of Theorem \ref{thm:IntTransf1}]
The kernel function of the integral operator $\mathcal{T}^\magn $ in \eqref{IntTransf1} is related to the exponential generating function involving the product of $e_{m,n}(u,v) = u^mv^n$ and $H_{m,n}^{\magn }(z;\bz)$.
Indeed, the functions $e_{m,n}(u,v) = u^mv^n$ form an orthogonal basis of the two-dimensional Bargmann-Fock space $\mathcal{F}^{2,\magn}(\C^2)$ whose
square norm is given by
   $$  \norm{e_{m,n}}_{L^{2,\magn}}^2= \left(\frac{\pi}{\magn}\right)^{2} \frac{m! n!}{\magn^{m+n}}.$$
While the polynomials $H_{m,n}^{\magn }(z;\bz)$, for varying $m$ and $n$, constitute an orthogonal basis of the Hilbert space $\Lnuc$.  
  Therefore, according to the principle described in the previous section, the corresponding kernel function
    \begin{align*}
T^\magn (z|u,v) = \left(\frac{\magn}{\pi}\right)^{3/2} \sum_{m,n=0}^{\infty} \frac{\magn ^{m+n} u^mv^n}{m!n!} \overline{H_{m,n}^{\magn }(z;\bz) }.
  \end{align*}
To conclude for Theorem \ref{thm:IntTransf1}, we need only to make use of the exponential generating function
 \cite{
 Gh13ITSF}
  \begin{equation}\label{genFctgHermite}
\sum_{m,n=0}^{\infty} \frac{u^mv^n}{m!n!} H_{m,n}^{\magn }(z;\bz) =  e^{\magn (uz +v \overline{z} -uv)} .
  \end{equation}
\end{proof}

\begin{proof}[of Theorem \ref{thm:IntTransf1b}]
The formula $\Gamma_{-i} \mathcal{R}^\magn_+  = \widetilde{\mathcal{F}}^{\nu/2}$  
is immediate from \eqref{IntTransf1} and \eqref{Fourier}. The formula \eqref{ConnFourier} follows by explicitly computing the action $\left( \mathcal{T}^\magn \right)^{-1} \Gamma_{-i} \mathcal{T}^\magn$.
Indeed, making use of \eqref{InverseT} giving the inverse of the transform $\mathcal{T}^\magn$ and by twice application of the integral formula \eqref{IntegFKernel} we obtain
   \begin{align*}
   \left( \mathcal{T}^\magn \right)^{-1} \Gamma_{-i} \mathcal{T}^\magn (\varphi )(\xi)
   &= \left(\frac{\magn}{\pi}\right)^2 \int_{\C} e^{-\magn |u|^2-i\magn u \bxi} \varphi (u) \left( \int_{\C} e^{-2\magn |z|^2 +\magn(\bxi -i \overline{u})z+\magn(\xi +i u)\bz } d\lambda(z) \right) d\lambda(u)
   \\ &= \left(\frac{\magn}{2\pi}\right) e^{\frac\magn 2 |\xi|^2 }\int_{\C} e^{-\frac\magn 2 |u|^2 -i\magn \left( \frac{u\bxi + \overline{u}\xi}{2} \right) } \varphi (u) d\lambda(u)\\
   &=   \widetilde{\mathcal{F}}^\magn (\varphi) (\xi) ,
    \end{align*}
    where $  \widetilde{\mathcal{F}}^\magn $ is as defined in \eqref{Fourier}. This completes the proof.
\end{proof}

\begin{proof}[of Theorem \ref{ThmBargmannWigner}]
The proof of $ \mathcal{B}^{2,\magn}  = \Gamma_{g_i}\mathcal{T}^\magn$ 
 follows by direct computation starting from $\Gamma_{g_i}\mathcal{T}^\magn \psi (z,w) = \mathcal{T}^\magn \psi \left(z+iw,z-iw\right)$ and next using the identity
$$ e^{-\magn(U-\xi)(V-\bxi)} = \rho_0^\magn\left(\Re\xi -\frac{U+V}{2}\right)\rho_0^\magn\left(\Im\xi -\frac{U-V}{2i}\right).$$
 Indeed, we obtain
  \begin{align*}
  \Gamma_{g_i}\mathcal{T}^\magn \psi (z,w) 
   & = \left(\frac{\magn}{\pi}\right)^{3/2} \int_{\C}  e^{-\magn\left(z+iw-\xi\right)\left(z-iw-\overline{\xi}\right)}  \psi(\xi) d\lambda(\xi)
  \\ & = \left(\frac{\magn}{\pi}\right)^{3/2} \int_{\C}   \rho_0^\magn\left(\Re\xi - z\right)\rho_0^\magn\left(\Im\xi -w\right)  \psi(\xi) d\lambda(\xi).
\end{align*}
which is noting but the integral transform  $\mathcal{B}^{2,\magn}$ given by \eqref{2dSBT}. 
To prove \eqref{ConnIntTransf1Wigner}, we rewrite $\mathcal{T}^\magn \mathcal{W}^\nu$ as
\begin{align*}
\mathcal{T}^\magn (\mathcal{W}^\nu (f))(z,w)&= \left(\frac{\magn}{\pi}\right)^{3/2} \int_{\C}  e^{-\magn\left(z+iw-\xi\right)\left(z-iw-\overline{\xi}\right)}  \mathcal{W}^\nu (f) (\xi) d\lambda(\xi)
\\&= \left(\frac{\magn}{\pi}\right)^{3/2} \int_{\C}  e^{-\magn x^2 +\magn(z+w)x -\magn(z-iy)(w+iy) }  \mathcal{W}^\nu (f) (\xi) d\lambda(\xi)
\end{align*}
with $\xi =x+iy\sim(x,y)$. By the definition \eqref{Wignerdf} of $\mathcal{W}^\nu$ and the Gaussian integral formula, we get
\begin{align*}
\mathcal{T}^\magn (\mathcal{W}^\nu (f))(z,w)
&= \left( \frac{1}{2\magn}\right)^{1/2}  \left(\frac{\magn}{\pi}\right)^{3/2}
 \int_{\R^2}  f\left(y-\frac{t}{2},y+\frac{t}{2}\right) e^{-\magn (z-iy)(w+iy) +\frac{\magn}{4}(z+w-it)^2 }   dydt
\\&= \left( \frac{1}{2\magn}\right)^{1/2}  \left(\frac{\magn}{\pi}\right)^{3/2}  e^{\frac{\magn}{4}(z+w)^2}
 \int_{\C}  f(\xi) e^{-\frac{\magn}{2} (-iz-w-\xi)(-iz+w-\bxi)} d\lambda(\xi)
 \\&= \left( \frac{1}{2\magn}\right)^{1/2}  e^{\frac{\magn}{4}(z+w)^2}
 \mathcal{T}^{\frac{\magn}2} (f)(-iz-w,-iz+w)
\end{align*}
thanks to the change of variables $X=y+\frac{t}{2}$ and $Y=y-\frac{t}{2}$ and the key observation that
\begin{align*}
e^{-\magn \left(z-i\frac{X+Y}{2}\right)\left(w+i\frac{X+Y}{2}\right) +\frac{\magn}{4}(z+w-i(X-Y))^2 }
=  e^{\frac{\magn}{4}(z+w)^2} e^{-\frac{\magn}{2} (-iz-w-\xi)(-iz+w-\bxi)}.
\end{align*}
To conclude, it suffices to see that $ \mathcal{T}^{\frac{\magn}2} (f)(-iz-w,-iz+w) =  \Gamma_{-ig_{i}}\mathcal{T}^{\frac{\magn}2} (f)(z,w)$, where
$g_{i}$ is the matrix in \eqref{gi}.
\end{proof}

\begin{proof}[of Theorem \ref{thm:IntTransf3}]
We apply the coherent states formalism described in Section \ref{s2}. Indeed, the kernel function in the integral transform $\mathcal{T}^\magn_{n,n'}$ defined by \eqref{IntTransf3},
  \begin{align*}
\mathcal{T}^\magn_{n,n'} (\psi)(z) := \left(\frac{(-1)^{n'} \magn}{\pi \sqrt{ n! n'! \magn^{n+n'} }}\right)
\int_{\C} e^{-\magn|\xi|^2 + \magn \bxi z }  H_{n,n'}^\magn(\xi-z,\bxi-\bz)   \psi(\xi)  d\lambda(\xi)
  \end{align*}
is in fact the exponential generating function involving the product of $H_{m,n}^{\magn }$ and $H_{m,n'}^{\magn }$, which for varying $m$, are special orthogonal bases of the generalized Bargmann-Fock spaces $\mathcal{F}^{2,\magn}_n(\C)$ and $\mathcal{F}^{2,\magn}_{n'}(\C)$, respectively. More precisely, we make use of Theorem \ref{thm:bilgenfct1}.
\end{proof}

\begin{proof}[of Theorem \ref{thm:bilgenfct1}] Notice for instance that it is easy to see that the Rodrigues' formula \eqref{gchpmu} infers
\begin{eqnarray} \label{vgchpmu}
H_{m,n}^\magn  (z,\bz )
=  (-1)^{m} \magn^{n} e^{\magn  z\bz }\dfrac{\partial ^{m}}{\partial \bz^{m}}  \left( \bz^{n} e^{-\magn  z \bz }\right).
 \end{eqnarray}
Subsequently, we can check that
\begin{eqnarray} \label{vgchpmuzw}
 H_{m,m'}^{\magn}( z - \xi , \bz - \overline{\xi}) =  (-1)^{m} e^{\magn  | z - \xi|^2 }  \dfrac{\partial ^{m}}{\partial \bz^{m}} \left( \magn^{m'} (\bz-\overline{\xi})^{m'}  e^{-\magn  | z - \xi|^2 }  \right) .
  \end{eqnarray}
Using successively the variant \eqref{vgchpmu} of the Rodrigues' formula as well as the generating function \eqref{GenFctm}, one gets
\begin{align*}
G_{m,m'}^\magn(t;z,w)
& = \sum\limits_{n=0}^{+\infty}\frac{t^n}{\magn^n n!} \left[   (-1)^{m} \magn^{n} e^{\magn  z\bz }\dfrac{\partial ^{m}}{\partial \bz^{m}}  \left( \bz^{n} e^{-\magn  z \bz }\right) \right] H_{n,m'}^\magn (w,\bar w)
\\ & =  (-1)^{m} e^{\magn  z\bz }  \dfrac{\partial ^{m}}{\partial \bz^{m}}
\left[ \magn^{m'} (\bw -t\bz)^{m'} e^{\magn t \bz w}  e^{-\magn  z \bz } \right].
\end{align*}
Now, if $t$ is assumed to belong to the unit circle, the above identity can be rewritten as
\begin{align*}
G_{m,m'}^\magn(t;z,w)
 =  (-t)^{m'}e^{\magn  z\bz }  e^{-\magn \overline{t}\bw ( z - t w) }
          (-1)^{m} \dfrac{\partial ^{m}}{\partial \bz^{m}} \left[ \magn^{m'} (\bz-\overline{t}\bw)^{m'}  e^{-\magn  | z - t w|^2 }  \right].
\end{align*}
In the right-hand side of the previous equality we recognize \eqref{vgchpmuzw}. Thus, the expression of $G_{m,m'}^\magn(t;z,w)$ reduces further to
the desired result \eqref{genfct1hh}.
\end{proof}

\section{Concluding remarks}

We have investigated some new integral transforms using the coherent states transform formalism. The first one connects $\Lnuc$ to the two-dimensional Bargmann-Fock space $\mathcal{F}^{2,\magn}(\C^2)$. This is the non-trivial version of the two-dimensional Segal-Bargmann transform.  The second transform connects any two generalized Bargmann-Fock spaces $\mathcal{F}_n^{2,\magn}(\C)$. These spaces are the $L^2$-eigenspaces of the magnetic Schr\"odinger operator $\Delta_\magn$ acting on $\Lnuc$. The obtained generating functions \eqref{thm:bilgenfct3HnHmn},
\eqref{genFctgHermite} and \eqref{genfct1hh}, for the UCHP,  have played a crucial role in obtaining these transforms.
As a consequence of Theorem \ref{thm:bilgenfct1}, we get
\begin{eqnarray} \label{genLag}
\sum\limits_{n=0}^{+\infty} \frac{ t^n }{n! \magn ^n  }  H_{m,n}^{\magn }(z,\bz ) H_{n,m}^{\magn }(w,\bw )
= (\magn t)^{m} m!   L^{(0)}_m(\magn |z -tw|^2)  e^{\magn  t w\bz }
\end{eqnarray}
which follows readily by taking $m=m'$ in \eqref{genfct1hh} and using $ H_{m,m}^{\magn }(\xi,\bar\xi) =  (-\magn )^m m! L^{(0)}_m(\magn |\xi|^2)$, where
$L^{(\gamma)}_m(x)$ denotes the Laguerre polynomials. The particular case of $z=w$ yields the identity
 \begin{eqnarray} \label{prob}
\sum\limits_{n=0}^{+\infty} \frac{ t^n }{n! \magn ^n  }  |H_{m,n}^{\magn }(z,\bz )|^2  &=
    m! (\magn t)^{m}   L^{(0)}_m(\magn |1 - t|^2 |z|^2)  e^{\magn  t |z|^2 }
\end{eqnarray}
for every $t$ in the unit circle and $z\in \C$. More particularly, we have
 \begin{eqnarray} \label{prob1}
\sum\limits_{n=0}^{+\infty} \frac{|H_{m,n}^{\magn }(z,\bz )|^2  }{n! \magn ^n  }   &=   m! \magn ^{m}  e^{\magn |z|^2 }.
\end{eqnarray}

 Using similar arguments as the ones used here, we are able to establish the following generating function involving the product $u^m t^n H_{m,n}^{\magn }(z,\bz ) H_{n,m'}^{\magn }(w,\bw )$.

\begin{theorem}\label{thm:bilgenfct2} For every $t$ in the unit circle, $|t|=1$, and complex numbers $u,z,w\in \C$ such that $\magn  |u|<1$, we have
\begin{eqnarray}\label{BilGen1}
\sum\limits_{m,n=0}^{+\infty} \frac{u^m t^n }{ m!n! \magn^n }  H_{m,n}^{\magn }(z,\bz ) H_{n,m'}^{\magn }(w,\bw )
&= (-\magn t)^{m'} ( \bz - \overline{t}\bw - u )^{m'} e^{\magn t \bz w + \magn u( z-t w) }.
\end{eqnarray}
\end{theorem}

\begin{proof}  {The identity \eqref{BilGen1} follows by twice application} of the generating function \cite[Proposition 3.4 (b)]{Gh13ITSF}
         \begin{eqnarray} \label{GenFctm}
         \sum\limits_{k=0}^{+\infty}\frac{u^k}{k!} H^\magn_{k,n}(z,\bz ) = \magn^{n} (\bz-u)^n  e^{\magn u z} .
        \end{eqnarray}
 It can also be obtained easily by combining \eqref{genfct1hh} and \eqref{GenFctm}.
\end{proof}

The closed expression we obtain  is new and not to be confused with the one obtained in \cite{IsmailZhang2017}
for the non-symmetry in the indices and the fact that $t$ belongs to the unit circle.
The same observation holds true for Theorem \ref{thm:bilgenfct1}.

Most of the obtained results in the framework of the UCHP can be rederived making use of rescaled version of the integral representation of the UCHP,
including the exponential generating function \eqref{genFctgHermite} as well as the generating function \eqref{GenFctm}.

\begin{theorem} For the scalar parameters $\mu>0$ and $\alpha,\beta\in \C$ such that $\alpha\beta >0$ and $\magn = \frac{\alpha\beta}{\gauss }$, we have
  \begin{eqnarray}\label{intHermite}
  H_{m,n}^{\magn}(z;\bz) = \left(\frac{\gauss}{\pi}\right)  (-\alpha)^m(\beta)^n
  \int_{\C} \xi^m \overline{\xi}^n e^{\frac{\alpha\beta}{\gauss} |z|^2 -\gauss|\xi|^2 +\alpha \xi \bz - \beta \bxi z} d\lambda(\xi).
  \end{eqnarray}
\end{theorem}

\begin{proof}
We present here a direct proof.
We use the integral representation of the Gaussian function $e^{-\frac{\alpha\beta}{\gauss }|z|^2}$,
  \begin{eqnarray}\label{intGauss}
  \int_{\C} e^{-\gauss |\xi|^2 +\alpha \xi\overline{z} - \beta \overline{\xi}z} d\lambda(\xi)
  = \left(\frac{\pi}{\gauss }\right)  e^{-\frac{\alpha\beta}{\gauss }|z|^2}.
  \end{eqnarray}
 The integral involved in left-hand side of \eqref{intGauss}
   converges uniformly in $z$ on every disc $D(0,r)$ of $\C$. Thus, by differentiate repeatedly both sides of \eqref{intGauss}, with respect to $z$ and $\bz$, we obtain the integral representation for the $  H_{m,n}^{\magn}(z;\bz)$ given by \eqref{intHermite}.
\end{proof}

\begin{remark}\label{rem:pcase}
The particular case of $\alpha=-\beta=i$ in the identity \eqref{intGauss} reads simply
  \begin{eqnarray}\label{intGaussclass}
  e^{-\frac{|z|^2}{\gauss}} = \left(\frac{\gauss}{\pi} \right) \int_{\C} e^{-\gauss|\xi|^2 + 2i\Re \xi \bz } d\lambda(\xi) ,
  \end{eqnarray}
and leads to the well-known fact that the Fourier transform reproduces the Gaussian function. By considering for example $\gauss=1$ and $\alpha =-\beta=i$,
the integral representation \eqref{intHermite} reduces further to the one
obtained in \cite[Theorem 5.1]{IsmailTrans2016} making use of the exponential generating function \eqref{genFctgHermite}.
\end{remark}

\noindent {\bf Acknowledgment:}

The authors 
acknowledge their indebtedness to Professor Ahmed Sebbar for reading the paper and for their precious remarks on it.
The activate members of "Intissar seminar" are gratefully acknowledged.


\begin{thebibliography}{9}
\bibitem{Abreu2010} {L.D. Abreu},
        {\it Sampling and interpolation in Bargmann-Fock spaces of polyanalytic functions}.
        Appl. Comput. Harmon. Anal. 29 (2010), no. 3, 287--302.
\bibitem{AbreuGossonMouayn2015}
{L.D. Abreu, P. Balazs, M. de Gosson, Z. Mouayn},
       {\it Discrete coherent states for higher Landau levels}.
       Ann. Physics 363 (2015) 337--353.
\bibitem{ABEG2015}
{F. Agorram, A. Benkhadra, A. El Hamyani, A. Ghanmi},
                 {\it Complex Hermite functions as Fourier-Wigner transform}.
                  Integral Transforms Spec. Funct.  27  (2016), no. 2, 94-100.
\bibitem{Bargmann1961}
{V. Bargmann},
        {\it On a Hilbert space of analytic functions and an associated integral transform}.
        Comm. Pure Appl. Math. 14 (1961) 187-214.
\bibitem{DG2018b}
{K. Diki, A. Ghanmi},
         {\it  Special integral transforms in the framework of holomorphic and the sliced hyperholomorphic Bargmann-Fock spaces}.
\bibitem{ElfardiGIS2018}
{A. El Fardi, A. Ghanmi, L. Imlal, M. Souid El Ainin},
        Analytic and arithmetic properties of the $(\Gamma,\chi)$-automorphic reproducing kernel function and associated Hermite-Gauss series.
        To appear in Ramanujan journal.
\bibitem{Folland1989}
{G.B. Folland},
       {\it Harmonic analysis in phase space}.
            Annals of Mathematics Studies, 122.
            Princeton University Press, Princeton, NJ; 1989.
\bibitem{Gazeau2009}
{J.-P. Gazeau},
    {\it Coherent states in quantum physics}.
     WILEY-VCH Verlag GmbH $\&$ Co. KGaA Weinheim 2009.
\bibitem{Gh08JMAA}
{ A. Ghanmi},
        {\it A class of generalized complex Hermite polynomials}.
        { J. Math. Anal. App.}  340 (2008), 1395-1406.
\bibitem{Gh13ITSF}
{A. Ghanmi},
         {\it Operational formulae for the complex Hermite polynomials $H_{p,q}(z, \bar z)$}.
        {Integral Transforms Spec. Funct.},  Volume 24, Issue 11 (2013) pp  884-895.  
\bibitem{Gh2017Mehler}
{A. Ghanmi},
         {\it Mehler's formulas for the univariate complex Hermite polynomials and applications}.
         To appear in Math. Methods Appl. Sci.
\bibitem{GrossMalliavin1996}
{L. Gross, P. Malliavin},
      {\it Hall's transform and the Segal-Bargmann map, It\^o´s stochastic calculus
      and probability}. ed. N. Ikeda et al., Springer, Berlin, 1996, 78--116.
\bibitem{Hall1994}
{B.C. Hall},
         {\it The Segal-Bargmann "coherent-state" transform for Lie groups}.
          J. Funct. Anal. 122 (1994) 103--151.
\bibitem{IsmailTrans2016}
{M.E.H. Ismail},
                             {\it Analytic properties of complex Hermite polynomials}.
                             Trans. Amer. Math. Soc. 368 (2016), no. 2, 1189-1210.
\bibitem{IsmailZhang2017}
{M.E.H. Ismail,  R. Zhang},
        A review of multivariate orthogonal polynomials.
        J. Egyptian Math. Soc. 25 (2017), no. 2, 91--110.
\bibitem{Shun-LongLuo1998}
{Shun-Long Luo},
        {\it On the Bargmann transform and the wigner transform}.
         Bull. London Math. Soc. 30 (1998) 413--418.
\bibitem{Mouayn2011}
{Z. Mouayn},
       {\it Coherent state transforms attached to generalized Bargmann spaces on the complex plane}.
       Math. Nachr. 284 (2011) no. 14-15, 1948--1954.
\bibitem{MouraoNQ2017}
{J. Mourao, J.P. Nunes,  T. Qian},
           {\it Coherent State Transforms and the Weyl Equation in Clifford Analysis}.
          J. Math. Phys. 58 (2017), no. 1, 013503, 12
\bibitem{Neretin1972}
{Y.A. Neretin},
         {\it  Lectures on Gaussian integral operators and classical groups}.
          EMS Series of Lectures in Mathematics. European Mathematical Society (EMS), Z\"urich, 2011.
\bibitem{Pool1966}
{J.C.T. Pool},
      {\it Mathematical aspect of Weyl correspondence}.
      J. Math. Phys. 7 (1966) 66--76.
\bibitem{Rainville71}
{E.D. Rainville},
        {\it Special functions},
        Chelsea Publishing Co., Bronx, N.Y., (1960).
\bibitem{Segal62}
{I. Segal},
        {\it Mathematical characterization of the physical vacuum for a linear Bose-Einstein field}.
         Illinois J. Math. 6 (1962) 500--523.
\bibitem{Thangavelu1998}
{S. Thangavelu},
          {\it Harmonic analysis on the Heisenberg group}.
          Progress in Mathematics, 159. Birkh\"auser Boston, Inc., Boston, MA, 1998.
\bibitem{Vasilevski2000}
{ N.L. Vasilevski},
               {\it Poly-Fock spaces}.
                Oper. Theory, Adv. App. 117 (2000)  371--386.
\end{thebibliography}
\end{document}